\documentclass{amsart}
\usepackage{amsmath, amscd, amssymb}
\begin{document}

\newcommand{\BB}{{\mathbb B}}
\newcommand{\CC}{{\mathbb C}}
\newcommand{\NN}{{\mathbb N}}
\newcommand{\ZZ}{{\mathbb Z}}

\newcommand{\cC}{{\mathcal C}}
\newcommand{\cF}{{\mathcal F}}
\newcommand{\cO}{{\mathcal O}}

\newcommand{\Gm}{{{\mathbb G}_m}}

\newtheorem{lemma}{Lemma}[section]
\newtheorem{theorem}[lemma]{Theorem}
\newtheorem{prop}[lemma]{Proposition}

\theoremstyle{definition}
\newtheorem{Def}[lemma]{Definition}

\theoremstyle{remark}
\newtheorem{Remark}[lemma]{Remark}

\title{Analytic relations on a dynamical orbit}
\author{Thomas Scanlon}
\thanks{Partially supported by a Templeton Infinity grant and NSF CAREER grant DMS-0450010}
\address{University of California, Berkeley \\
Department of Mathematics \\
Evans Hall\\
Berkeley, CA 94720-3840 \\
USA}
\email{scanlon@math.berkeley.edu}
\maketitle
\begin{abstract}
Let $(K,\vert\cdot\vert)$ be a complete discretely valued field and $f:\BB_1(K,1) \to \BB(K,1)$ a nonconstant analytic map from the unit back to itself.  We assume that $0$ is an attracting fixed point of $f$.  Let $a \in K$ with $\lim_{n \to \infty} f^n(a) = 0$ and consider the orbit $\cO_f(a) := \{ f^n(a) : n \in \NN \}$.  We show that if $0$ is a \emph{superattracting} fixed point, then every irreducible analytic subvariety of $\BB_n(K,1)$ meeting $\cO_f(a)^n$ in an analytically Zariski dense set is defined by equations of the form $x_i = b$ and $x_j = f^\ell(x_k)$.  When $0$ is an attracting, non-superattracting point, we show that all analytic relations come from algebraic tori.
\end{abstract}

\section{Introduction}
In this paper, we study a local analytic version of the so-called dynamical Mordell-Lang problem (see, for instance,~\cite{GTZ} or~\cite{Zh} in which the conjecture is explicitly raised).  Here, one considers an algebraic variety $X$ defined over some field $K$, a regular self-map $f:X \to X$, and a point $a \in X(K)$ and then studies possible algebraic relations on the forward orbit of $a$ under $f$, $\cO_f(a) := \{ f^n(a) : n \in \NN \}$.  That is, one seeks to describe the intersections $Y(K) \cap \cO_f(a)^m$ where $Y \subseteq X^m$ is an algebraic subvariety of some Cartesian power of $X$.  In the classical situation, $X$ is itself a semi-abelian variety and $f$ is an endomorphism of $X$.  The forward orbit $\cO_f(a)$ is then a subset of the subgroup generated by $a$ and it is known that the only irreducible algebraic varieties which can meet a finitely generated subgroup of a semi-abelian variety in a Zariski dense set are themselves translates of sub-semi-abelian varieties.  For the more general dynamical Mordell-Lang problem, one asks merely that the irreducible algebraic varieties meeting $\cO_f(a)^m$ in a Zariski dense set simply be $f$-periodic.  We believe that this conclusion begs the question as to the form of the $f$-periodic varieties and in this paper we take the further step of explicitly describing these exceptional varieties.

While the dynamical Mordell-Lang problem concerns algebraic varieties, the approaches to its solution to date are based on local arguments in which the algebraic dynamical system is regarded as a $p$-analytic dynamical system.  Here we drop the hypothesis that the analytic dynamical system considered comes from a rational function.  Specifically, after a change of variables, we consider convergent analytic functions $f(x) = \sum_{n=1}^\infty b_i x^i$ over some complete DVR and then study the possible analytic relations on the forward orbits $\cO_f(a)$ where $a$ is close enough to zero.  In the case that $0$ is an attracting, but not super-attracting point (by which we mean that $0 < |b_1| < 1$), then all such analytic relations come from suitably deformed algebraic tori.  When $0$ is a super-attracting point ($b_1 = 0$), then the analytic relations are all defined by equations of the form $f(x_i) = x_j$ and $f^k(a) = x_\ell$.  

The key to our proof is a transformation of analytic relations on $\cO_f(a)^n$ into linear relations on certain subsets of the value group.  We then use a theorem of van den Dries and G\"{u}nayd{\i}n~\cite{vdDG} to give a refined description of the solutions to these equations.

When the analytic dynamical system comes from a rational function $f$, then our results fully describe the possible algebraic relations on the dynamical orbit $\cO_f(a)$ when $\lim f^n(a)$ is a superattracting fixed point (relative to some valuation).  In case, $\lim f^n(a)$ is merely an attracting fixed point, then our results limit the possible algebraic relations, but in general, one would need to decide which of the deformed tori are algebraic and we do not address that question here.

\section{Generalities on analytic dynamics}
In this section we set out our notation and conventions about rigid analytic dynamics and prove some basic reductions.

Throughout this paper we work over a complete discretely valued field $(K,\vert\cdot\vert)$.  On occasion, we wish to write the valuation additively as
$v:K \to \ZZ \cup \{ \infty \}$ and to relax the hypothesis that the valuation on $K$ be discrete.   For each real number $r$ and natural number $n$ we write $\BB_n(K,r)$ for the closed polydisc $\{(x_1,\ldots,x_n) \in K^n : |x_i| \leq r \text{ for } i \leq n \}$.  Note in particular that $\BB_1(K,1)$ is the valuation ring.  While one often writes the valuation ring as $\cO_K$, when thinking of $\BB_1(K,1)$ as a ring, we shall write $R$ so as to avoid confusion with the notation for dynamical orbits.

We recall some of the basic results on rigid analysis over $K$.  The reader may wish to consult~\cite{BGR} for details, though we follow a different notation and most of what we do would make sense for other formalisms, such as those of Berkovich~\cite{Ber} or Huber~\cite{Hu}.

For each natural number $n$, the Tate algebra $K \langle x_1, \ldots, x_n \rangle$
consists of those formal power series over $K$ in the variables $x_1, \ldots, x_n$ convergent on $\BB_n(K,1)$.  That is,
$$K \langle x_1, \ldots, x_n \rangle := \{ \sum f_I x^I \in K [[x_1,\ldots,x_n]] :
\sum |f_I| < \infty \}$$
Of course, in our ultrametric setting, $\sum |f_I| < \infty$ if and only on $\lim_{|I| \to \infty} |f_I| = 0$.    Here, we have employed a multi-index notation which will appear later.  For $I = (i_1, \ldots, i_n) \in {^n}\NN$ and an $n$-tuple $(x_1,\ldots,x_n)$, we write $x^I$ for $\prod_{j=1}^n x_j^{i_j}$.

One might also consider the Tate algebra over restricted rings of coefficients.  We shall have occasion to look at $R \langle x_1, \ldots x_n \rangle = K \langle x_1, \ldots, x_n \rangle \cap R [[x_1, \ldots, x_n ]]$.  It is a basic result of Tate that $R \langle x_1, \ldots, x_n \rangle$ is Noetherian. Indeed, without the hypothesis that the valuation is discrete, it is still the case that $K \langle x_1, \ldots, x_n \rangle$ is Noetherian.  There is only one point in our arguments where the hypothesis that $R$ is a DVR is used, but it appears to be essential.

In general, one can make sense of dynamical systems in any category and with respect to any semigroup, but we shall restrict attention to analytic dynamical systems with respect to $\NN$, or with respect to $\NN^n$ but derived from one on $\NN$.  For us, a dynamical system is given by a self-map $f:X \to X$ where $X$ is a rigid analytic space and $f$ is analytic.  The dynamics are understood through the iteration of $f$.   That is, we have an induced map $\NN \times X \to X$ given by $(m,x) \mapsto f^m(x)$.  A morphism of dynamical systems $h:(X,f) \to (Y,g)$ is given by an analytic map $h:X \to Y$ for which the following diagram commutes.

$$\begin{CD} X @>{f}>> X \\ @V{h}VV  @VV{h}V \\ Y @>{g}>> Y \end{CD}$$

For a point $x \in X$ we define the forward orbit of $x$ under $f$ to be
$$\cO_f(x) := \{ f^n(x) : n \in \NN \}$$
Here, what is meant by ``point'' depends on the specific choice of our definition of analytic space.  In most of what we are doing, $X$ will be the unit ball of $K$ and one may think of $x$ as an element of the maximal ideal of $R$.

Note that if $h:(X,f) \to (Y,g)$ is a morphism of dynamical systems, $x \in X$, and $y = h(x)$, then $h$ maps $\cO_f(x)$ onto $\cO_g(y)$.

If $x \in X$ is a fixed point of $f:X \to X$, then $df_x:T_xX \to T_xX$ is another dynamical system.  We say that $x$ is an \emph{attracting} fixed point if every eigenvalue of $df_x$ has absolute value less than $1$.  We say that $x$ is a \emph{superattracting} fixed point if $df_x = 0$.  Of course, one should really further refine these notions in terms of the rank of $df_x$, but as we shall focus on the one dimensional case, such a distinction is irrelevant.

Suppose for the moment that $f:X \to X$ is an analytic dynamical system, $a \in X$, and $\overline{a} = \lim_{n \to \infty} f^n(a)$ exists.  In particular, $\overline{a}$ is a fixed point of $f$.  Possibly at the cost of replacing $a$ by $f^N(a)$ for some $N \gg 0$ and performing a change of variables, we may represent this situation in terms of a analytic dynamical system on a polydisc for which $\overline{a}$ corresponds to the origin.  In the one dimensional situation, we can further simplify the presentation.

\begin{lemma}
\label{localform}
Let $f:X \to X$ be an analytic self-map.  We assume that $\dim X = 1$, $a \in X$, and $\overline{a} = \lim_{n \to \infty} f^n(a)$ is a smooth point.  Then there is an open subspace $U \subseteq X$ with $\overline{a} \in U$, an analytic map $g:\BB_1(K,1) \to \BB_1(K,1)$, and an analytic isomorphism $h:U \to \BB_1(K,1)$ such that $f(U) \subseteq U$ and $h:(U,f \upharpoonright U) \to (\BB_1(K,1),g)$ is a map of dynamical systems. Moreover, $f^N(a) \in U$ for $N \gg 0$.
\end{lemma}

\begin{proof}
As $\overline{a}$ is a smooth point, we can find a neighborhood $V$ of $\overline{a}$ isomorphic to $\BB_1(K,1)$ via a map taking $\overline{a}$ to $0$. As $\overline{a}$ is a limit of an $f$-orbit, it must be an $f$-fixed point which is not repelling.  In particular, possibly at the cost of restricting $V$ to a smaller coordinate neighborhood $U$, $f$ maps $U$ back to itself.   Let $h:U \to \BB_1(K,1)$ express $U$ as a coordinate neighborhood with $h(\overline{a}) = 0$.  Set $g := h^{-1} \circ f \circ h$.
\end{proof}

By Lemma~\ref{localform} we may restrict attention to analytic dynamical systems on the unit disc.  Such dynamical systems may be taken to have a very simple form.  Our result on a canonical form for dynamical systems on $\BB_1(K,1)$ with $0$ as an attracting fixed point is needed only in the non-superattracting case, but we include a statement in the superattracting case to complete the picture.

\begin{lemma}
\label{formalmult}
Suppose that $f:\BB_1(K,1) \to \BB_1(K,1)$ is an analytic dynamical system on the unit disc and that $f(0) = 0$, and $\lambda := f'(0)$ has $0 < |\lambda| < 1$.  That is, $0$ is a non-superattracting, attracting fixed point.  Then there is a positive real number $\epsilon \leq 1$ and an analytic isomorphism $h:\BB_1(K,\epsilon) \to \BB_1(K,\epsilon)$ for which $f(\BB_1(K,\epsilon)) \subseteq \BB_1(K,\epsilon)$ and $h:(\BB_1(K,\epsilon),f\upharpoonright \BB_1(K,\epsilon)) \to (\BB_1(K,\epsilon),x \mapsto \lambda x)$ is an isomorphism of dynamical systems.
\end{lemma}

\begin{proof}
Write $f(x) = \sum_{i=1}^\infty f_i x^i$ where $f_1 = \lambda$.  Since $f(\BB_1(K,1)) \subseteq \BB_1(K,1)$ we have $1 \geq ||f||_1 =: \sup \{ |f(x)| : x \in \BB_1(K,1) \} = ||f||_{\text{Gau\ss}} =: \sup \{ |f_i| : i \in \NN \}$.  Thus, our hypothesis that $f$ maps the unit ball back to itself implies $f \in R \langle x \rangle$.  This implies that for any positive $\epsilon < 1$ and $x \in \BB_1(K,\epsilon)$ we have
$|f(x)| = |\sum f_i x^i| \leq \sup \{ |f_i x^i| : i \in \ZZ_+ \} \leq \sup \{ |f_i| \epsilon^i : i \in \ZZ_+ \} \leq \epsilon$.  Hence, $f$ restricts to a self-map of $\BB_1(K,\epsilon)$.

We look now for the required $h$.  Write $h(x) = \sum_{j=1}^\infty h_j x^j$.  We will take $h_1 = 1$.   If $h$ is going to work for us, then we need $h(f(x)) = \lambda h(x)$. Let us compute $h \circ f(x)$.

\begin{eqnarray*}
h \circ f(x) & = & \sum_{j=1}^\infty h_j (\sum_{i=1}^\infty f_i x^i)^j \\
& = & \sum_{j=1}^\infty \sum_{I \in {^j}\ZZ_+} h_j (\prod_{\ell=1}^j f_{I_\ell}) x^{\sum I_\ell} \\
& = & \sum_{N=1}^\infty (\sum_{\begin{array}{c} I \in {^{<\omega}} \ZZ_+ \\ \sum_\ell I_\ell = N \end{array}}
h_{|I|} \prod f_{I_\ell}) x^N
\end{eqnarray*}

Comparing the coefficients of $x^N$, we see that we wish to have
\begin{eqnarray*} \lambda h_N  & = &  \sum_{\begin{array}{c} I \in {^{<\omega}} \ZZ_+ \\ \sum_\ell  I_\ell = N \end{array}}
h_{|I|} \prod f_{I_\ell} \\
& = & h_N \lambda^N + \sum_{\begin{array}{c} I \in  {^{<\omega}} \ZZ_+ \\ \sum_\ell  I_\ell = N  \\ |I| < N  \end{array}}
h_{|I|} \prod f_{I_\ell}
\end{eqnarray*}

Thus, we may recursively solve for $h_N$ as
$$h_N := \frac{1}{\lambda - \lambda^N}  \sum_{\begin{array}{c} I \in  {^{<\omega}} \ZZ_+ \\ \sum_\ell  I_\ell = N  \\ |I| < N  \end{array}}
h_{|I|} \prod f_{I_\ell}$$

As $|\lambda| < 1$ and we have $|f_i| \leq 1$ for all $i$ and $f_1 = \lambda$, we see that $|h_N| \leq |\lambda|^{1-N}$.  Thus, if we take $\epsilon = |\lambda|$, then
$h$ is convergent on $\BB_1(K,\epsilon)$ and defines an isomorphism between $(\BB_1(K,\epsilon),f)$ and $(\BB_1(K,\epsilon),x \mapsto \lambda x)$.
\end{proof}

As we noted above, a similar canonical form exists in the superattracting case, but we do not need this result for the sequel and it only holds in the case that $K$ has characteristic zero.

\begin{prop}
Suppose that $K$ has characteristic zero.  Suppose that $f:\BB_1(K,1) \to \BB_1(K,1)$ is a nonconstant analytic self-map with $f(0) = f'(0) = 0$.  That is, $0$ is a superattracting fixed point for $f$.  Then there are $\lambda \in \BB_1(K,1)$, $M \geq 2$, $\epsilon > 0$, and an isomorphism of dynamical systems $h:(\BB_1(K,\epsilon),f) \to (\BB_1(K,\epsilon),x \mapsto \lambda x^M)$.
\end{prop}
\begin{proof}
Let $M$ be minimal with $f^{(M)}(0) \neq 0$.  Write $f(x) = \sum_{i=M}^\infty f_i x^i$ and set $\lambda := f_M = \frac{f^{(M)}(0)}{M!}$.  As in Lemma~\ref{formalmult}, we see that $|f_i| \leq 1$ for all $i \in \NN$.  Moreover, if we write $h(x) = \sum_{j=1}^\infty h_j x^j$ taking $h_1 = 1$ and attempt to solve $h \circ f (x)  = \lambda (h(x))^M$, then expression for $h \circ f(x)$ computed in the course of the proof of Lemma~\ref{formalmult} is still valid.  This time, since $M > 1$, we need to expand the righthand side of the equation.

\begin{eqnarray*}
\lambda (h(x))^M & = & \lambda (\sum_{j=1}^\infty h_j x^j)^M \\
 & = & \lambda \sum_{j_1=1}^\infty \cdots \sum_{j_M=1}^\infty \prod_{\ell=1}^M h_{j_\ell} x^{\sum j_\ell} \\
 & = & \sum_{N=M}^\infty (\sum_{\begin{array}{c} J \in {\ZZ_+}^M \\ \sum J_\ell = N \end{array}}
 \lambda \prod_{\ell=1}^M h_{J_\ell}) x^N
\end{eqnarray*}

Equating the coefficients of $x^N$, we must solve
$$ \sum_{\begin{array}{c} I \in {^{<\omega}} \ZZ_+ \\ \sum_\ell I_\ell = N \end{array}}
h_{|I|} \prod f_{I_\ell} = \sum_{\begin{array}{c} J \in {\ZZ_+}^M \\ \sum J_\ell = N \end{array}}
 \lambda \prod_{\ell=1}^M h_{J_\ell}$$

In this case, the largest index for $h$ appearing in the equations occurs on the righthand side with $\lambda M h_{N - (M-1)}$.  (Recall that $h_1 = 1$).  Thus, we may formally solve for $h$ by taking
$$h_n := \frac{1}{M \lambda} [(\sum_{\begin{array}{c} I \in {^{<\omega}} \ZZ_+ \\ \sum_\ell I_\ell = n + (M-1) \end{array}}
h_{|I|} \prod f_{I_\ell}) - (\sum_{\begin{array}{c} J \in {\{1,\ldots,n-1\}}^M \\ \sum J_\ell = n+(M-1) \end{array}}
\lambda \prod_{\ell=1}^M h_{J_\ell})]$$

If we take $\epsilon := |M \lambda|$, then $h$ defines an analytic automorphism of $\BB_1(K,\epsilon)$.
\end{proof}

\section{Special analytic varieties}
In this section we study the special analytic varieties.  That is, we look at those analytic varieties which must contain dense sets of points from dynamical orbits.  These analytic varieties come in two forms: those coming from graphs of iterates of the dynamical system itself and those coming from pullbacks of algebraic tori.  When considering nonsuperattracting dynamics, the special analytic varieties of the first kind are instances of special varieties of the second kind.

Throughout this section, we work over the complete valued field $(K,\vert\cdot\vert)$
with corresponding additive valuation $v$ and $f:\BB_1(K,1) \to \BB_1(K,1)$ is a nonconstant analytic self-map of the unit disc for which $0$ is an attracting fixed point. Write $f(x) = \sum_{i=M}^\infty f_i x^i$ where $\lambda = f_M \neq 0$. We fix some $a \in K$ with $|a| < |\lambda|$.

Before we introduce our special varieties, we compute the valuations of the elements of $\cO_f(a)$.

\begin{lemma}
\label{iterval}
 If $M = 1$, then for each natural number $n$ we have $v(f^n(a)) = v(a) + n v(\lambda)$.  If $M \geq 2$,  then for each natural number $n$ we have
$v(f^n(a)) = M^n v(a) + \frac{M^n - 1}{M-1} v(\lambda) = M^n [v(a) + \frac{v(\lambda)}{M -1}] + \frac{v(\lambda)}{1 - M}$.
\end{lemma}
\begin{proof}
Note that we have a background hypothesis that $f$ maps the unit disc back to itself.  Hence, $|f_i| \leq 1$ for all $i$.  It follows from the ultrametric inequality, that if $|b| < |\lambda|$, then $|f(b)| = |\lambda b^M|$.

In each of the cases, the initial instance of $n = 0$ is trivial to verify.  Moreover, by induction in each case we have $|f^n(a)| < \lambda$. For the inductive step, in the nonsuperattracting case we have by induction $v(f^{n+1}(a)) = v(f(f^n(a))) = v(\lambda f^n(a)) = v(\lambda) + v(f^n(a)) = v(\lambda) + v(a) + n v(\lambda) = v(a) + (n+1) v(\lambda)$.  In the superattracting case, we have $v(f^{n+1}(a)) = v(f(f^n(a))) = v(\lambda (f^n(a))^M) = v(\lambda) + M v(f^n(a)) = v(\lambda) + M (M^n v(a) + \frac{M^n - 1}{M-1} v(\lambda))  = M^{n+1} v(a) + \frac{M^{n+1} - 1}{M - 1} v(\lambda)$.
\end{proof}

With Lemma~\ref{iterval} in place, we may introduce our first class of special varieties.

\begin{Def}
Let $n \in \ZZ_+$ and let $D \subseteq K \langle x_1, \ldots, x_n \rangle$ be an ideal generated by a set of the form $\{ x_i - f^{m_i}(a) : i \in I \} \cup \{ x_j - f^{\ell_{(j,k)}}(x_k) : (j,k) \in J \}$ where $I \subseteq \{1, \ldots, n \}$,
$J \subseteq \{ 1, \ldots, n \}^2$, and each $m_i$ and $\ell_{(j,k)}$ is a natural numbers.  We write $X_D = V(D)$ for the analytic variety defined by $D$ and say that
$X_D$ is an \emph{iterational special variety}.
\end{Def}

It helps to choose canonical defining equations for iterational special varieties.
Of course, it is possible that the above equations are inconsistent.  In this case, $1$ generates $D$.  Otherwise, given an iterational variety $X$, let $I' := \{ i \in \{ 1, \ldots, n \} : x_i = f^m(a) \text{ holds on } X \text{ for some } m \in \NN \}$.  For each $i \in M$, let $m_i \in \NN$ be the unique such $m$.  Note that since $|f^m(a)| < |f^\ell(a)|$ for $m > \ell$ by Lemma~\ref{iterval}, there can be only one such $m$ if $X \neq \varnothing$.  Let $J'$ be the set of pairs $(j,k)$ for which $x_j = f^\ell(x_k)$ holds on $X$ for some $\ell$.  As before, for each such pair there is a unique $\ell_{(j,k)}$ for which this is true.  By this uniqueness, these equations must include the defining equations for $X$.

\begin{prop}
\label{iterint}
Let $X$ be a nonempty iterational special variety.  Let $I'$, $\{ m_i : i \in I' \}$, $J'$, and $\{ \ell_{(j,k)} : (j,k) \in J' \}$ be the data computed above for $X$.  Then $X$ is irreducible and $X(K) \cap \cO_f(a)^n = \{ (f^{t_1}(a), \ldots, f^{t_n}(a)) :
t_i = m_i \text{ for } i \in I', t_k = t_j + \ell_{(j,k)} \text{ for } (j,k) \in J' \}$ is analytically Zariski dense in $X$.
\end{prop}
\begin{proof}
Let $G := \{ k \in \{ 1, \ldots, n \} : k \notin I' \text{ and } (\nexists j) (j,k) \in J'\}$.  It should be clear that the coordinate ring of $X$ is $K \langle \{ x_k : k \in G \} \rangle$ so that in particular, $G$ is irreducible being isomorphic to the polydisc $\BB_{|G|}(K,1)$.

The calculation of $X \cap \cO_f(a)^n$ is clear.
\end{proof}

\begin{Remark}
It bears noting, that by Lemma~\ref{iterval}, the map $\NN^n \to \cO_f(a)^n$ given by $(t_1,\ldots,t_n) \mapsto (f^{t_1}(a),\ldots,f^{t_n}(a))$ is a bijection.  Thus, $X \cap \cO_f(a)^n$ corresponds to a translate of a diagonal submonoid of $\NN^n$, that is a submonoid defined by equalities of the form $t_i = 0$ and $t_j = t_k$.
\end{Remark}

The other class of special varieties we shall encounter also correspond to translates of submonoids of $\NN^n$ but in this case every submonoid of the form $G \cap \NN^n$ where $G$ is a subgroup of $\ZZ^n$ may appear.

\begin{Def}
We assume now that $M = 1$.  Let $h:(\BB_1(K,|\lambda|),f) \to (\BB_1(K,|\lambda|),x \mapsto \lambda x)$ be the analytic isomorphism of Lemma~\ref{localform}. We say that an analytic variety $X \subseteq \BB_n(K,|\lambda|)$ is a \emph{deformed torus} if there is a connected algebraic group $T \leq \Gm_K^n$ and a point $\xi \in \cO_{\lambda x}(h(a))^n$ for which $h(X) = \xi T \cap \BB_n(K,|\lambda|)$ where we have written ``$h$'' for the induced map on $\BB_n(K,|\lambda|)$ and $\Gm_K$ for the multiplicative group scheme over $K$.
\end{Def}

\begin{Remark}
The notion of a deformed torus depends on the choice of the analytic automorphism $h$ and the reader should be aware that when we speak of an analytic variety being a deformed torus, we have in mind a fixed choice of $h$.
\end{Remark}

As algebraic tori are defined by character equations, deformed tori also have simple canonical defining equations in terms of pullbacks of character equations.

\begin{prop}
\label{deftorint}
Suppose that $X$ is a deformed torus.  Let $\xi = h((f^{t_1}(a)),\ldots,f^{t_n}(a))$ and $T \leq \Gm_K$ be an algebraic torus witnessing that $X$ is a deformed torus via $X = h^{-1}(\xi T \cap \BB_n(K,|\lambda|))$.  Let
$A \in M_{m \times n}(\ZZ)$ be a matrix for which $T$ is defined by $\prod_{i=1}^n x_i^{A_{j,i}} = 1$ ($j \leq m$).  Then $X \cap \cO_f(a)^n = \{ (f^{s_1}(a),
\ldots, f^{s_n}(a)) :  \sum_{i=1}^n A_{j,i} (s_j - t_j) = 0  \text{ for } j \leq m \}$
and is analytically Zariski dense in $X$.
\end{prop}
\begin{proof}
In the case that we are already dealing with $f(x) = \lambda x$ so that $h$ is the identity function and $X = \xi T$, then the result is clear.  In general, as $h$ is an isomorphism of dynamical systems, $h$ induces a bijection between
$X \cap \cO_f(a)^n$ and $\xi T \cap \cO_{\lambda x}(h(a))^n$ respecting the action of $\NN^n$.
\end{proof}

We are now in a position to state our main theorem, though its proof will require a few more lemmata.

\begin{theorem}
\label{mainthm}
If $M = 1$, then an irreducible analytic variety $X \subseteq \BB_n(K,|\lambda|)$ meets $\cO_f(a)^n$ in an analytically Zariski dense set if and only if $X$ is a deformed torus.  If $M \geq 2$, then an irreducible analytic variety meets $\cO_f(a)^n$ in an analytically Zariski dense set if and only if it is iterational.
\end{theorem}

With Propositions~\ref{iterint} and~\ref{deftorint} we have already noted that the right-to-left implications hold.  Thus, it remains for us to prove the left-to-right implications.

Let us start with the easiest of the three remaining lemmata.

\begin{lemma}
\label{closetozero}
In Theorem~\ref{mainthm}, it suffices  to assume that there is some positive $\epsilon \leq |\lambda|$ with $X \cap \BB_n(K,\epsilon) \cap \cO_f(a)^n$ Zariski dense in $X$.
\end{lemma}
\begin{proof}
We work by induction with the case of $n = 0$ being trivial. Let $X$ be an irreducible variety containing a dense set of points from $\cO_f(a)^n$.  Let $\epsilon \leq |\lambda|$ be a positive number. As $\lim_{m \to \infty} f^m(a) = 0$, there is a positive integer $N$ such that $f^m(a) \in \BB_1(K,\epsilon)$ for $m > N$. Thus, $X \cap \cO_f(a)^n =  (X \cap \BB_n(K,|\epsilon|) \cap \cO_f(a)^n) \cup
\bigcup_{j=1}^n \bigcup_{i=0}^N [X \cap \pi_j^{-1}\{f^i(a)\}] \cap \cO_f(a)^n$ where $\pi_j$ is the projection onto the $j^\text{th}$ coordinate.  As $X$ is irreducible, if
$X \cap \BB_n(K,\epsilon) \cap \cO_f(a)^n$ is not Zariski dense in $X$, then $X$ must be equal to one of $X \cap \pi_j^{-1}(f^i(a))$.  The coordinate projection onto the complementary coordinates defines an isomorphism between $X$ and an analytic subvariety of $\BB_{n-1}(K,|\lambda|)$ and takes $\cO_f(a)^n$ to $\cO_f(a)^{n-1}$.  Hence, by induction, $X$ already has the requisite form.
\end{proof}

The next lemma is the only place in our argument where the hypothesis that $K$ is discretely valued is used.

\begin{lemma}
\label{minval}
Let $n \in \ZZ_+$ be a positive integer and $G = \sum_I g_I x^I \in R \langle x_1, \ldots, x_n \rangle$ a convergent power series in $n$ variables over $R$.  There is a
finite set $\cF$ of multi-indices such that for any $b \in \BB_n(K,1)$ for some $I \in \cF$ and for all $J \in \NN^n$ we have $|g_I b^I| \geq |g_J b^J|$.
\end{lemma}

\begin{proof}
We work by induction on $n$ with the case of $n = 0$ being trivial.  In the case of $n+1$, write $G = \sum_{i=0}^\infty g_i(x_1,\ldots,x_n) x_{n+1}^i$.  By Noetherianity of $R \langle x_1, \ldots, x_n \rangle$, the ideal $(\{g_i:i \in \NN \})$ is generated by
$\{ g_i : i \leq N \}$ for some natural number $N$.  Let $\cF_0, \ldots, \cF_N$ be the
finite sets of multi-indices given by induction for $g_i$ with $i \leq N$.  Let $\cF := \{ Ii : I \in \cF_i, i \leq N \}$.  Let us check that this choice of $\cF$ works.  Let $b = (b_1,\ldots,b_{n+1}) \in \BB_{n+1}(K,1)$.  Let $J \in \NN^{n+1}$ and write $J = J'j$ where $J' \in \NN^n$ and $j \in \NN$.   We may express $g_j$ as an $R \langle x_1, \ldots, x_n \rangle$-linear combination of $g_0, \ldots, g_N$.  The monomial $g_J x^J$ is the product of a monomial in $g_j$ with $x_{n+1}^j$.  Hence, from the expression of $g_j$ we recover an expression $g_J x^J = \sum_{i=0}^N h_i g_{J_{i}i} x^{J_{i}i}$ for appropriate choices of multi-indices $J_i$ and $h_i \in R \langle x_1, \ldots, x_n \rangle$. By the ultrametric inequality and the fact that that $|h_i(b_1,\ldots,b_n)| \leq 1$, we have $|g_J b^J| \leq |g_{J_{i}i} b^{J_{i}i}|$ for some $i \leq N$.  By the definitions of $\cF_i$ and $\cF$, we find some $I' \in \cF_i$  (and hence $I = I'i \in \cF$) such that $|g_J b^J| \leq |g_{J_{i}i} b^{J_{i}i}| \leq |g_{I'i} b^{I'i}| = |g_I b^I|$.
\end{proof}

In the nonsuperattracting case, we know enough already to finish the proof.

For the superattracting case we need the Mann property isolated in~\cite{vdDG}.

\begin{lemma}
\label{mannthm}
Let $\Gamma \leq \CC^\times$ be a finitely generated subgroup of the multiplicative group of the complex numbers.  If $L(x) = \sum_{i=1}^n c_i x_i \in \CC[x_1,\ldots,x_n]$ is a nonzero linear polynomial and $c \in \CC$ is any complex number, then the set
$\{ (\gamma_1,\ldots,\gamma_n) \in \Gamma^n : L(\gamma_1,\ldots,\gamma_n) = c \}$ is a finite union of sets defined by equations of the form $x_i = \gamma$ and $x_j = \delta x_k$ for $\gamma$ and $\delta$ in $\Gamma$.
\end{lemma}
\begin{proof}
We work by induction on $n$ with $n = 0$ being trivial.
The main theorem of~\cite{vdDG}  asserts that $\Gamma$ has the \emph{Mann property}: there only finitely many nondegenerate solutions to the above equation where by \emph{degenerate} we mean that $\sum_{i \in I} c_i x_i = 0$ for some $I \subsetneq \{ 1, \ldots, n \}$.

Note that the Mann property implies immediate that for a homogeneous equation, that is, when $c=0$, that the solutions in $\Gamma^n$ are either degenerate or fall into one of finitely many sets of the form $\{ (\delta \cdot \gamma_1, \ldots, \delta \gamma_n) : \delta  \in \Gamma \}$.  This last set is defined by the equations $x_j = \gamma_j \gamma_1^{-1} x_1$ ($1 < j \leq n$).

In the following equation, for $I \subseteq \{ 1, \ldots, n \}$ we write $I' := \{1, \ldots, n \} \smallsetminus I$ for the complement of $I$ in $\{1, \ldots, n \}$.

We have
\begin{eqnarray*}
 \{ (\gamma_1,\ldots,\gamma_n) \in \Gamma^n : L(\gamma_1,\ldots,\gamma_n) = c \}
 &= & \text{ finite set of nondegenerate solutions } \cup \\
 && \bigcup_{I \subsetneq \{ 1, \ldots, n \}} \{ (\gamma_1, \ldots, \gamma_n) \in \Gamma^n : \sum_{i \in I} c_i \gamma_i = 0  \\ && \text{ and } \sum_{i \in I'} c_i \gamma_i = c \}
\end{eqnarray*}

By induction, each of the sets $\{ (\gamma_j)_{j \in I} \in \Gamma^I : \sum_{j \in I} c_j \gamma_j = 0 \}$ and $\{ (\gamma_j)_{j \in I'} \in \Gamma^{I'} : \sum_{j \in I} c_j \gamma_j = c \}$ have the desired form, and, therefore, so does there product.
\end{proof}

\noindent
{\bf Proof of Theorem~\ref{mainthm}}:
We work by induction on $n$ where the case of $n = 0$ is trivial.  Suppose now that
$X \subseteq \BB_{n+1}(K,|\lambda|)$ is an irreducible analytic variety containing a Zariski dense set of points from $\cO_f(a)^{n+1}$.  Unless $X = \BB_{n+1}(K,|\lambda|)$, which is already a deformed torus, we can find a nontrivial $G \in R \langle x_1, \ldots, x_{n+1} \rangle$ vanishing on $X$.  Write $G = \sum g_I x^I$ and let $\cF$ be the finite set of multi-indices given by Lemma~\ref{minval}.  It follows from the ultrametric inequality that if $b \in \BB_{n+1}(K,1)$ and $G(b) = 0$, then there is some pair $I \neq J \in \cF$ with $|g_I b^I| = |g_J b^J| \neq 0$.  Written additively, taking $b = (f^{t_1}(a),\ldots,f^{t_{n+1}}(a))$, using Lemma~\ref{iterval} we would have in the non-superattracting case

\begin{equation}
\label{nonsaeq}
\sum_{i=1}^{n+1} v(\lambda) (I_i - J_i) t_i = v(g_J) - v(g_I) + (|J| - |I|) v(a)
\end{equation}

and in the superattracting case we have
\begin{equation}
\label{saeq}
\sum_{i=1}^{n+1} (I_i - J_i) M^{t_i} [v(a) + \frac{v(\lambda)}{M -1}] =  v(g_J) - v(g_I) + (|J| - |I|) \frac{v(\lambda)}{1 - M}
\end{equation}

In case $M = 1$ and presuming that there are any points from $\cO_f(a)^{n+1}$ satisfying the equation $G(b) = 0$ and having corresponding valuations satisfying Equation~\ref{nonsaeq}, there is a deformed torus [or really, a finite union of such as we have insisted that deformed tori be irreducible and in Equation~\ref{nonsaeq} there might be a common divisor of $\{ I_i - J_i : i \leq n+1 \}$] $Y$ such that $Y \cap \cO_f(a)^{n+1}$ consists exactly of the points $(f^{t_1}(a),\ldots,f^{t_{n+1}}(a))$ for the solutions $(t_1,\ldots,t_{n+1}) \in \NN^{n+1}$ to Equation~\ref{nonsaeq}.

In case $M > 1$, then by Lemma~\ref{mannthm} applied to $\Gamma = M^\ZZ$, the set of $(n+1)$-tuples of natural numbers satisfying Equation~\ref{saeq} is a finite union of sets defined by equations of the form $t_i = A$ or $t_j = t_k + B$ for natural numbers $A$ and $B$.  By Lemma~\ref{iterint}, such a set exactly corresponds to $\cO_f(a)^{n+1} \cap Y$ for some iterational variety $Y$.

In either case, intersecting with $X$ and using the fact that $X$ is irreducible, we may assume that $X \subseteq Y$ for some such irreducible deformed torus or iterational variety (depending on whether $M = 1$ or $M > 1$).  By induction, we may assume that $X$ projects onto a Zariski dense subset of $\BB_n(K,|\lambda|)$ for each co-ordinate projection.  Since $Y$ admits a finite projection, it follows that $X = Y$. \hspace{\fill} $\Box$

\begin{Remark}
The methods employed here give useful information in the case that $f:X \to X$ is an analytic self-map where $X$ has higher dimension greater and one considers dynamical orbits near an attracting fixed point, but the results are less definitive.  Unlike the case considered in this paper where $\dim X = 1$, it need not be the case that every relation allowed by the equations on the valuations is actually realized by an analytic variety.  We shall return to these issues in a sequel.   
\end{Remark}

%

\end{document}